\documentclass[a4paper,12pt, times]{amsart}
\usepackage[utf8x]{inputenc}

\usepackage{graphicx}
\usepackage{amssymb}
\usepackage{amsmath}
\usepackage{amsfonts}
\usepackage[mathscr]{eucal}
\usepackage{enumerate}
\usepackage{multicol}
\usepackage{times}
\usepackage[varg]{txfonts}
\usepackage{hyperref}
\newtheorem{thm}{Theorem}
\newtheorem{lem}[thm]{Lemma}
\newtheorem{cor}[thm]{Corollary}

\begin{document}

\title[Characterization of derivations]{Characterization of derivations through their actions on certain elementary functions
}

%\titlerunning{Characterization of derivations}        

\author[E.~Gselmann]{Eszter Gselmann}

%\authorrunning{E.~Gselmann} 
\address{
Department of Analysis, 
Institute of Mathematics, 
University of Debrecen,
P. O. Box: 12., 
Debrecen, 
H--4010, 
Hungary}

\email{gselmann@science.unideb.hu} 
\date{\today }
% The correct dates will be entered by the editor

\thanks{This research has been supported by the Hungarian Scientific Research Fund
(OTKA) Grant NK 81402 
and by the T\'{A}MOP 4.2.4.A/2-11-1-2012-0001 (Nemzeti Kiv\'{a}l\'{o}s\'{a}g Program --
Hazai hallgat\'{o}i, illetve kutat\'{o}i szem\'{e}lyi t\'{a}mogat\'{a}st biztos\'{i}t\'{o} rendszer kidolgoz\'{a}sa \'{e}s m\H{u}k\"{o}dtet\'{e}se
konvergencia program) project implemented
through the EU and Hungary co-financed by the European Social Fund.}

\maketitle

\begin{abstract}
The main aim of this note is to provide characterization theorems concerning real derivations. 
Among others the following implication will be verified: Assume that $\xi\colon \mathbb{R}\to \mathbb{R}$ 
is a given differentiable function and for the additive function $d\colon \mathbb{R}\to \mathbb{R}$, the mapping 
\[
 x\longmapsto d\left(\xi(x)\right)-\xi'(x)d(x)
\]
is regular (e. g. measurable, continuous, locally bounded). Then $d$ is a sum of a derivation and a linear function. 
\keywords{derivation, linear function, elementary function }
% \PACS{PACS code1 \and PACS code2 \and more}
 \subjclass{39B82,  39B72}
\end{abstract}

%\linenumbers

\section{Introduction}

Throughout this paper $\mathbb{N}$ denotes the set of the positive integers, further $\mathbb{Z}, \mathbb{Q}$, and
$\mathbb{R}$ stand for the set of the integer, the set of the rational and the set of the real numbers, respectively. 

The aim of this work is to prove characterization theorems on derivations as well as on linear functions. 
Therefore, firstly we have to recall some definitions and auxiliary results. 

A function $f:\mathbb{R}\rightarrow\mathbb{R}$ is called an \emph{additive} function
if,
\[
f(x+y)=f(x)+f(y)
\]
holds for all $x, y\in\mathbb{R}$. 

We say that an additive
function $f:\mathbb{R}\rightarrow\mathbb{R}$ is a \emph{derivation} if
\[
f(xy)=xf(y)+yf(x)
\]
is fulfilled for all $x, y\in\mathbb{R}$.

Clearly, the identically zero function is a real derivation. 
It is rather difficult to give another example, since the following statements are valid concerning real derivations. 
If $f\colon \mathbb{R}\to \mathbb{R}$ is a real derivation, then $f(x)=0$ holds for all 
$x\in \mathrm{algcl}(\mathbb{Q})$ (the algebraic closure of the rationals). Further, 
if $f\colon \mathbb{R}\to \mathbb{R}$ is a real derivation and $f$ is measurable or bounded (above or below) on a set of positive Lebesgue 
measure, then $f$ is identically zero. 
Despite of this very pathological behavior, there exist  non identically zero derivations in $\mathbb{R}$, 
see Kuczma \cite[Theorem 14.2.2.]{Kuc09}. 

The additive function $f\colon\mathbb{R}\to\mathbb{R}$ is termed to be a \emph{linear function} if 
$f$ is of the form 
\[
 f(x)=f(1)\cdot x 
\qquad 
\left(x\in\mathbb{R}\right). 
\]

It is easy to see from the above definition that every derivation
$f:\mathbb{R}\rightarrow\mathbb{R}$ satisfies equation
\[
\tag{$\ast$}\label{ast} f(x^{k})=kx^{k-1}f(x)
\quad
\left(x\in\mathbb{R}\setminus\left\{0\right\}\right)
\]
for arbitrarily fixed $k\in\mathbb{Z}\setminus\left\{0\right\}$.
Furthermore, the converse is also true, in the following sense:
if $k\in\mathbb{Z}\setminus\left\{0, 1\right\}$ is fixed and
an additive function $f:\mathbb{R}\rightarrow\mathbb{R}$ satisfies (\ref{ast}), 
then $f$ is a derivation, see e.g.,
Jurkat \cite{Jur65}, Kurepa \cite{Kur64}, and
Kannappan--Kurepa \cite{KanKur70}.

Concerning linear functions, Jurkat \cite{Jur65} and,
independently, Kurepa \cite{Kur64} proved that every additive function
$f:\mathbb{R}\rightarrow\mathbb{R}$ satisfying
\[
f\left(\frac{1}{x}\right)=\frac{1}{x^{2}}f(x)
\quad
\left(x\in\mathbb{R}\setminus\left\{0\right\}\right)
\]
has to be linear.

In \cite{NisHor68} A.~Nishiyama and S.~Horinouchi investigated
additive functions $f:\mathbb{R}\rightarrow\mathbb{R}$
satisfying the additional equation
\[
f(x^{n})=cx^{k}f(x^{m})
\quad
\left(x\in\mathbb{R}\setminus\left\{0\right\}\right),
\]
where $c\in\mathbb{R}$ and $n, m, k\in\mathbb{Z}$ are arbitrarily
fixed.

Henceforth we will say that the
function in question is \emph{regular} on its domain, if at least 
one of the following statements are fulfilled. 
\begin{enumerate}[(i)]
\item locally bounded;
\item continuous;
\item measurable in the sense of Lebesgue. 
\end{enumerate}

Concerning rational functions F.~Halter-Koch and L.~Reich proved similar result for derivations as well as linear functions, 
see \cite{HalRei01},\cite{HalRei00}. These results were strengthened in \cite{Gse13} in the following way. 

\begin{thm}\label{thm1}
 Let $n\in\mathbb{Z}\setminus\left\{0\right\}$ and 
$\left(\begin{array}{cc}
a&b\\
c&d
\end{array}
\right)\in\mathbf{GL}_{2}(\mathbb{Q})$ be such that 
\begin{enumerate}[--]
 \item if $c=0$, then $n\neq 1$;
\item if $d=0$, then $n\neq -1$. 
\end{enumerate}
Let further $f, g\colon\mathbb{R}\to\mathbb{R}$ be additive functions and define the function 
$\phi$ by 
\[
 \phi(x)=f\left(\frac{ax^{n}+b}{cx^{n}+d}\right)-\frac{x^{n-1}g(x)}{\left(cx^{n}+d\right)^{2}} 
\qquad 
\left(x\in\mathbb{R}, \,
cx^{n}+d\neq 0\right). 
\]
Let us assume $\phi$ to be regular. 
Then, the functions $F, G\colon\mathbb{R}\to\mathbb{R}$ defined by 
\[
 F(x)=f(x)-f(1)x \quad \text{and} 
\quad 
G(x)=g(x)-g(1)x
\quad
\left(x\in\mathbb{R}\right)
\]
are derivations. 
\end{thm}

Roughly speaking the above cited papers dealt with a special case of the following problem. 
Assume that $\xi\colon \mathbb{R}\to \mathbb{R}$ 
is a given differentiable function and for the additive function $d\colon \mathbb{R}\to \mathbb{R}$, the mapping 
\[
 x\longmapsto d\left(\xi(x)\right)-\xi'(x)d(x)
\]
is regular on its domain. It is true that in case 
$d$ admits a representation
\[
 d(x)=\chi(x)+d(1)\cdot x 
\quad 
\left(x\in \mathbb{R}\right), 
\]
where $\chi\colon \mathbb{R}\to \mathbb{R}$ is a real derivation?

In view of the above results, in case $n\in\mathbb{Z}\setminus\left\{0\right\}$ and 
$\left(\begin{array}{cc}
a&b\\
c&d
\end{array}
\right)\in\mathbf{GL}_{2}(\mathbb{Q})$
and the function $\xi$ is 
\[
 \xi(x)=\dfrac{ax^{n}+b}{cx^{n}+d}
\quad 
\left(x\in \mathbb{R}, cx^{n}+d\neq 0\right), 
\]
then the answer is \emph{affirmative}. 
The main aim of this note is to extend this result to other classes of elementary functions such as 
the exponential function, the logarithm function, the trigonometric functions and the hyperbolic functions. 
Concerning such type of investigations, we have to remark the paper of Gy.~Maksa (see \cite{Mak13}), where the previous 
problem was investigated under the supposition that the mapping 
\[
 x\longmapsto d\left(\xi(x)\right)-\xi'(x)d(x)
\]
is identically zero. 

\section{The main result}

Our main result is contained in the following. 

\begin{thm}\label{thm2}
 Assume that for the additive function 
$d\colon \mathbb{R}\to \mathbb{R}$ the mapping $\varphi$ defined by 
\[
 \varphi(x)=d\left(\xi(x)\right)-\xi'(x)d(x)
\]
is regular. Then the function $d$ can be represented as 
\[
 d(x)=\chi(x)+d(1)\cdot x 
\quad 
\left(x\in \mathbb{R}\right), 
\]
where $\chi\colon \mathbb{R}\to \mathbb{R}$ is a derivation, 
in any of the following cases
\begin{multicols}{2}
\begin{enumerate}[(a)]
 \item \[\xi(x)=a^{x}\]
\item \[\xi(x)=\cos(x)\]
\item \[\xi(x)=\sin(x)\]
\item \[\xi(x)=\cosh(x)\]
\item \[\xi(x)=\sinh(x). \]
\end{enumerate}
\end{multicols}
\end{thm}
\begin{proof}
 \begin{enumerate}[{Case} (a)]
  \item Let $a\in \mathbb{R}\setminus\left\{1\right\}$ be an arbitrary positive real number and suppose that the 
mapping $\varphi$ defined by 
\[
 \varphi(x)=d\left(a^{x}\right)-a^{x}\ln(a)d(x) 
\quad 
\left(x\in \mathbb{R}\right)
\]
is regular. 
A easy calculation shows that 
\[
 \varphi(2x)-2a^{x}\varphi(x)=d\left((a^{x})^{2}\right)-2a^{x}d\left(x\right)
\quad 
\left(x\in \mathbb{R}\right), 
\]
that is 
\[
 \varphi\left(2\log_{a}(u)\right)-2u\varphi\left(\log_{a}(u)\right)=
d(u^{2})-2ud(u)
\quad 
\left(u\in ]0, +\infty[\right). 
\]
Due to the regularity of the function $\varphi$, the mapping 
\[
 ]0, +\infty[\ni u\longmapsto \varphi\left(2\log_{a}(u)\right)-2u\varphi\left(\log_{a}(u)\right)
\]
is regular, too. Thus by Theorem \ref{thm1}, 
\[
 d(x)=\chi(x)+d(1)\cdot x 
\qquad 
\left(x\in \mathbb{R}\right), 
\]
where the function $\chi\colon \mathbb{R}\to \mathbb{R}$ is a derivation. 
\item Assume now that for the additive function $d\colon \mathbb{R}\to \mathbb{R}$, the mapping 
$\varphi$ defined on $\mathbb{R}$ by 
\[
 \varphi(x)=d\left(\cos(x)\right)+\sin(x)d(x) 
\qquad 
\left(x\in \mathbb{R}\right)
\]
is regular. 
If so, then
\[
 \dfrac{\varphi(2x)-4\cos(x)\varphi(x)+d(1)}{2}=
d\left(\cos^{2}(x)\right)-2\cos(x)f\left(x\right)
\]
holds for all $x\in \mathbb{R}$. 
Let now $u\in ]-1, 1[$ and write $\mathrm{arccos}(u)$ in place of $x$ to get 
\[
 \dfrac{\varphi(2\mathrm{arccos}(u))-4u\varphi(\mathrm{arccos}(u))+d(1)}{2}=
d(u^{2})-2ud(u). 
\]
Again, due to the regularity of the function $\varphi$, the mapping 
\[
 ]-1, 1[ \ni u\longmapsto \dfrac{\varphi(2\mathrm{arccos}(u))-4u\varphi(\mathrm{arccos}(u))+d(1)}{2}
\]
is regular, as well. Therefore, Theorem \ref{thm1} again implies that 
\[
 d(x)=\chi(x)+d(1)\cdot x 
\qquad 
\left(x\in \mathbb{R}\right), 
\]
is fulfilled with a certain real derivation $\chi\colon \mathbb{R}\to \mathbb{R}$. 
\item Suppose that for the additive function $d$, the mapping 
\[
 \varphi(x)=d\left(\sin(x)\right)-\cos(x)d(x)
\qquad 
\left(x\in \mathbb{R}\right)
\]
is regular. 
In this case 
\begin{multline*}
 \varphi\left(x-\frac{\pi}{2}\right)
=
d\left(\sin \left(x-\frac{\pi}{2}\right)\right)-\cos\left(x-\frac{\pi}{2}\right)d\left(x-\frac{\pi}{2}\right)
\\
-d\left(\cos(x)\right)-\sin(x)d(x)+\sin(x)d\left(\frac{\pi}{2}\right), 
\end{multline*}
that is, 
\[
 -\varphi\left(x-\frac{\pi}{2}\right)+\sin(x)d\left(\frac{\pi}{2}\right)=
d\left(\cos(x)\right)+\sin(x)d(x)
\qquad
\left(x\in \mathbb{R}\right). 
\]
 In view of Case (b) this yields that the function $d$ has the desired representation as stated. 
\item Assume the $d\colon \mathbb{R}\to \mathbb{R}$ is an additive function and the mapping 
\[
 \varphi(x)=d\left(\cosh(x)\right)-\sinh(x)d(x)
\qquad 
\left(x\in \mathbb{R}\right)
\]
is regular. The additivity of $d$ and some addition formula of the $\cosh$ function furnish 
\[
 \dfrac{\varphi(2x)-4\cosh(x)\varphi(x)+d(1)}{2}=
d\left(\cosh^{2}(x)\right)-2\cosh(x)d\left(\cosh(x)\right) 
\qquad 
\left(x\in \mathbb{R}\right). 
\]
Let now $u\in ]1, +\infty[$ arbitrary and put $x=\mathrm{arcosh}(u)$ into the previous identity to get 
\[
 \dfrac{\varphi(2\mathrm{arcosh}(u))-4u\varphi(\mathrm{arcosh}(u))+d(1)}{2}=
d(u^{2})-2ud(u). 
\]
Since the function $\varphi$ is regular, the mapping 
\[
 ]1, +\infty[\ni u\longmapsto \dfrac{\varphi(2\mathrm{arcosh}(u))-4u\varphi(\mathrm{arcosh}(u))+d(1)}{2}
\]
will also be regular. Therefore, Theorem \ref{thm1} implies again the desired decomposition of the function 
$d$. 
\item Finally, assume the $d\colon \mathbb{R}\to \mathbb{R}$ is an additive function so that 
\[
 \varphi(x)=d\left(\sinh(x)\right)-\cosh(x)d(x)
\qquad 
\left(x\in \mathbb{R}\right)
\]
is regular. 
Let $x, y\in \mathbb{R}$ be arbitrary, then 
\begin{multline*}
\varphi(x+y)= d\left(\sinh(x+y)\right)-\cosh(x+y)d(x+y)
\\=
d\left(\sinh(x)\cosh(y)\right)+d\left(\sinh(y)\cosh(x)\right)
\\-\left[\sinh(x)\sinh(y)+\cosh(x)\cosh(y)\right]d(x+y)
\\=
d\left(\sinh(x)\cosh(y)\right)+d\left(\sinh(y)\cosh(x)\right)
-\sinh(x)\sinh(y)d(x+y)
\\-\cosh(x)d(x)\cosh(y)
-\cosh(x)\cosh(y)d(y)
\end{multline*}
If we use the definition of the function $\varphi$, after some rearrangement, we arrive at 
\begin{multline*}
 \varphi(x+y)-\varphi(x)\cosh(y)-\varphi(y)\cosh(x)
\\
=
d\left(\sinh(x)\cosh(y)\right)+d\left(\sinh(y)\cosh(x)\right)
-\sinh(x)\sinh(y)d(x+y)\\
-\cosh(y)d\left(\sinh(x)\right)-\cosh(x)d\left(\sinh(y)\right)
\end{multline*}
for all $x, y\in \mathbb{R}$. If we replace here $y$ by $-y$, 
\begin{multline*}
 \varphi(x-y)-\varphi(x)\cosh(y)-\varphi(-y)\cosh(x)
\\
=
d\left(\sinh(x)\cosh(y)\right)-d\left(\sinh(y)\cosh(x)\right)
+\sinh(x)\sinh(y)d(x-y)\\
-\cosh(y)d\left(\sinh(x)\right)+\cosh(x)d\left(\sinh(y)\right)
\end{multline*}
can be concluded, where we have also used that the function $\cosh$ is even and the function 
$\sinh$ is odd. 
Adding this two identities side by side, 
\begin{multline*}
 \Phi(x, y)= 2d\left(\sinh(x)\cosh(y)\right)
\\+\sinh(x)\sinh(x)\left[d(x-y)-d(x+y)\right]
-2\cosh(y)d\left(\sinh(x)\right)
\end{multline*}
for any $x, y\in \mathbb{R}$, where 
\begin{multline*}
 \Phi(x, y)= 
\varphi(x+y)-\varphi(x)\cosh(y)-\varphi(y)\cosh(x)
\\+\varphi(x-y)-\varphi(x)\cosh(y)-\varphi(-y)\cosh(x)
\quad 
\left(x, y\in \mathbb{R}\right). 
\end{multline*}
If we put $x=\mathrm{arsinh}(1)$, we get that 
\[
 \dfrac{\Phi\left(\mathrm{arsinh}(1), y\right)+2\cosh(y)d(1)}{2}
=
d\left(\cosh(y)\right)-\sinh(y)d(y)
\quad 
\left(y\in \mathbb{R}\right). 
\]
Due to the regularity of the function $\varphi$, the mapping 
\[
 \mathbb{R}\ni y\longmapsto \dfrac{\Phi\left(\mathrm{arsinh}(1), y\right)+2\cosh(y)d(1)}{2}
\]
is regular, too. Hence, Case (d) yields the desired form of the function $d$. 
\end{enumerate}
\end{proof}

In what follows, we would like to extend the list of the functions appearing in the previous statement. 
Therefore we prove the following. 

\begin{lem}\label{lem3}
 Let $d\colon \mathbb{R}\to \mathbb{R}$ be an additive function, 
$I\subset \mathbb{R}$ be a nonvoid open interval and 
$\xi\colon I\to \mathbb{R}$ be a continuously differentiable function so that 
the derivative of the function $\xi^{-1}\colon \xi(I)\to \mathbb{R}$ is nowhere zero. 
The mapping 
\[
 I \ni x\longmapsto d(\xi(x))-\xi'(x)d(x)
\]
is regular if and only if the mapping 
\[
 \xi(I)\ni u\longmapsto d(\eta(u))-\eta'(u)d(u)
\]
is regular, where $\eta=\xi^{-1}$. 
\end{lem}
\begin{proof}
 Assume that for the additive function $d$, we have that the mapping 
\[
 \varphi(x)=d(\xi(x))-\xi'(x)d(x) 
\qquad 
\left(x\in I\right)
\]
is regular. 
Let now $u\in \xi(I)$ and put $\xi^{-1}(u)$ in place of $x$ to get 
\[
 -\left(\xi^{-1}\right)'(u)\varphi(\xi^{-1}(u))=
d\left(\xi^{-1}(u)\right)-\left(\xi^{-1}\right)'(u)d(u). 
\]
Due to the regularity of $\varphi$, the mapping appearing in the left hand side is also regular, as stated. 
\end{proof}

In view of Theorem \ref{thm2} and Lemma \ref{lem3}, we immediately obtain the following theorem.

\begin{cor}\label{cor4}
  Assume that for the additive function 
$d\colon \mathbb{R}\to \mathbb{R}$ the mapping $\varphi$ defined by 
\[
 \varphi(x)=d\left(\xi(x)\right)-\xi'(x)d(x)
\]
is regular. Then the function $d$ can be represented as 
\[
 d(x)=\chi(x)+d(1)\cdot x 
\quad 
\left(x\in \mathbb{R}\right), 
\]
where $\chi\colon \mathbb{R}\to \mathbb{R}$ is a derivation, 
in any of the following cases
\begin{multicols}{2}
\begin{enumerate}[(a)]
 \item \[\xi(x)=\ln(x)\]
\item \[\xi(x)=\mathrm{arccos}(x)\]
\item \[\xi(x)=\mathrm{arcsin}(x)\]
\item \[\xi(x)=\mathrm{arcosh}(x)\]
\item \[\xi(x)=\mathrm{arsinh}(x). \]
\end{enumerate}
 
\end{multicols}
\end{cor}

\section{Stability of derivations}

As a starting point of the proof of the main result of this section
the theorem of Hyers will be used. Originally this
statement was formulated in terms of functions that
are acting between Banach spaces, see Hyers \cite{Hye41}.
However, we will use this theorem only in the particular case
when the domain and
the range are the set of reals.
In this setting we have the following.

\begin{thm}
Let $\varepsilon\geq 0$ and suppose that the function
$f:\mathbb{R}\rightarrow\mathbb{R}$ fulfills the inequality
\[
\left|f(x+y)-f(x)-f(y)\right|\leq \varepsilon
\]
for all $x,y\in\mathbb{R}$. Then there exists an additive
function $a:\mathbb{R}\rightarrow\mathbb{R}$ such that
\[
\left|f(x)-a(x)\right|\leq \varepsilon
\]
holds for arbitrary $x\in\mathbb{R}$.
\end{thm}

In other words, Hyers' theorem states that if a function $f\colon \mathbb{R}\to \mathbb{R}$ fulfills 
the inequality appearing above, then it can be represented as 
\[
 f(x)=a(x)+b(x) 
\qquad 
\left(x\in \mathbb{R}\right), 
\]
where $a\colon \mathbb{R}\to \mathbb{R}$ is an additive and 
$b\colon \mathbb{R}\to \mathbb{R}$ is a bounded function. Moreover, for all 
$x \in \mathbb{R}$, we also have $\left|b(x)\right|\leq \varepsilon$. 

With the aid of Hyers' theorem and the results of the previous section, the following 
stability type result can be proved. Concerning stability properties of derivations 
the interested reader may consult Badora \cite{Bad06} and Boros--Gselmann \cite{BorGse10}. 

\begin{thm}
 Let $\varepsilon>0$ be arbitrarily fixed, 
$f\colon \mathbb{R}\to \mathbb{R}$ be a function and suppose that 
\begin{enumerate}[(A)]
 \item for all $x, y\in \mathbb{R}$ we have 
\[
 \left|f(x+y)-f(x)-f(y)\right|\leq \varepsilon. 
\]
\item the mapping 
\[
 x\longmapsto f\left(\xi(x)\right)-\xi'(x)f(x)
\]
is locally bounded on its domain, where the function $\xi$ is one of the functions 
\begin{multicols}{2}
 \begin{enumerate}[(a)]
 \item \[a^{x}\]
\item \[\cos(x)\]
\item \[\sin(x)\]
\item \[\cosh(x)\]
\item \[\sinh(x)\]
\item \[\ln(x)\]
\item \[\mathrm{arccos}(x)\]
\item \[\mathrm{arcsin}(x)\]
\item \[\mathrm{arcosh}(x)\]
\item \[\mathrm{arsinh}(x). \]
 \end{enumerate}
\end{multicols}

\end{enumerate}
Then there exist $\lambda\in \mathbb{R}$ and a real derivation $\chi\colon \mathbb{R}\to \mathbb{R}$ 
such that 
\[
 \left|f(x)-\left[\chi(x)+\lambda\cdot x\right]\right|\leq \varepsilon
\]
holds for all $x\in \mathbb{R}$. 
\end{thm}
\begin{proof}
 Due to assumption (A), we immediately have that 
\[
 f(x)=a(x)+b(x)
\qquad 
\left(x\in \mathbb{R}\right), 
\]
where $a\colon \mathbb{R}\to \mathbb{R}$ is an additive and 
$b\colon \mathbb{R}\to \mathbb{R}$ is a bounded function.
If we use supposition (B), from this we get that the mapping
\[
 x\longmapsto \left[a(\xi(x))-\xi'(x)a(x)\right] +\left[b(\xi(x))-\xi'(x)b(x)\right]
\]
is locally bounded. From this however the local boundedness of the function 
\[
 x\longmapsto a(\xi(x))-\xi'(x)a(x)
\]
can be deduced. In view of the previous statements (see Theorem \ref{thm2} and Corollary \ref{cor4}), 
\[
 a(x)=\chi(x)+a(1)\cdot x 
\qquad 
\left(x\in \mathbb{R}\right)
\]
is fulfilled for any $x\in \mathbb{R}$, where $\chi\colon \mathbb{R}\to \mathbb{R}$ is a certain real derivation. 
For the function $f$ this means that there exists $\lambda \in \mathbb{R}$ and a real derivation 
$\chi\colon \mathbb{R}\to \mathbb{R}$ so that 
\[
f(x)=\chi(x)+\lambda\cdot x+b(x)
\qquad 
\left(x\in \mathbb{R}\right), 
\]
or equivalently 
\[
 \left|f(x)-\left[\chi(x)+\lambda\cdot x\right]\right|\leq \varepsilon
\]
is staisfied for any $x\in \mathbb{R}$. 
\end{proof}

\subsubsection*{Acknowledgements}
This note is dedicated to two important men in my life, to Fr\'{e}di and to P\'{e}ter.

%\bibliographystyle{spmpsci}
%\bibliography{stability_derivation}

\end{document}